\documentclass[preprint,12pt]{elsarticle}

\usepackage{amssymb}

\usepackage{amsfonts}
\usepackage{amsmath}
\usepackage{amsthm}
\usepackage{amssymb}
\usepackage{graphicx}

\usepackage{color}

\numberwithin{equation}{section}
\newtheorem{theorem}{Theorem}[section]
\newtheorem{lemma}{Lemma}[section]

\usepackage{comment}
\usepackage[utf8]{inputenc}
\usepackage{tikz}

\DeclareMathOperator\erf{erf}

\begin{document}

\begin{frontmatter}



\title{Mathematical model of thermal phenomena of closure electrical contact with Joule heat source and nonlinear thermal coefficients}

\author[a,b]{Julieta Bollati}

\author[a,b]{Adriana C. Briozzo\corref{cor1}}

\author[c,d]{Stanislav N. Kharin}

\author[c,d,e]{Targyn A. Nauryz}

\address[a]{Depto. Matem\'atica, FCE, Univ. Austral, Paraguay 1950, Rosario, Argentina}

\address[b]{CONICET, Buenos Aires, Argentina}

\address[c]{Kazakh-British Technical University, Almaty, Kazakhstan}

\address[d]{Institute of Mathematics and Mathematical Modeling, Almaty, Kazakhstan}

\address[e]{Narxoz University, Almaty, Kazakhstan}

\cortext[cor1]{abriozzo@austral.edu.ar, (+54) 341 5223000}

\begin{abstract}
The mathematical model describing the dynamics of closed contact heating which involves vaporization of the metal when instantaneous explosion of micro-asperity occurs is presented through a Stefan type problem. The temperature field for metallic vaporization zone is introduced as heat resistance that decreases linearity. Temperature fields for liquid and solid phases of the metal described by spherical heat equations with nonlinear thermal coefficients and Joule heat source have to be determined as well as the free boundaries. Joule heating component depends on space and time variable when alternating current is considered. Solution method of the problem based on similarity variable transformation is applied which enables us to reduce the problem to an ordinary differential equations and nonlinear integral equations. Existence and uniqueness of the integral equations are proved by using fixed point theorem in Banach space. 
\end{abstract}



\end{frontmatter}

\section{Introduction}	

The classical Stefan problem is a fundamental mathematical model that describes the solidification or melting of a material with a moving phase boundary. In the classical formulation of the Stefan problem, the heat transfer is assumed to occur solely by conduction, and the material properties are considered constant throughout the process. The problem involves the analysis of the temperature distribution and the movement of the phase boundary as a material undergoes phase change. The Stefan problem finds applications in several fields, including materials science, solidification processes, heat transfer, and phase change phenomena  \cite{AlSo1993,Ca1984,CaJa1959,Cr1984,Gu2017,Lu1991,Ru1971,Sh1967,St1889-1,Ti1959}. In particular, classical Stefan problems in electrical contacts involving the analysis of heat transfer and phase change phenomena were studied in the papers  \cite{Kh2014,KhNaMi2020,KhNa2021,SaErNaNo2018}.

However, in many practical applications, such as phase change phenomena in engineering, biology, and geophysics, the assumptions of the classical Stefan problem may not hold. Nonlinear effects, non-uniform material properties, and additional physical processes can significantly influence the solidification or melting process. These deviations from the classical assumptions give rise to the study of non-classical Stefan problems. Non-classical Stefan problem for semi-infinite and moving material with given boundary temperature and heat flux condition on the fixed face are widely studied in  \cite{BoBr2021,BrNaTa2007,BrTa2006-2,KuSiRa2020B,SiKuRa2019A,Ta2011}. Non-classical Stefan problems in electrical contacts extend the classical formulation by considering additional factors and physical processes. These include the presence of electrical currents flowing through the contact interface, nonlinear behavior of the contact materials, and the influence of non-uniform temperature and current distributions. The consideration of these factors leads to a more realistic representation of electrical contact phenomena, enabling a more accurate analysis of the system's behavior and performance. Nonlinear Stefan problems in electrical contacts arise when the material properties of the contact materials exhibit nonlinear behavior, such as temperature-dependent electrical conductivity or thermal conductivity \cite{KhNa2021-A,NaKh2022}.

   Investigation of electrical arc phenomena occurring in contacts of electrical apparatus during their opening and closure is very important for increasing the reliability, durability and fail-safety of their operation and for reducing electrical erosion. 
   The tendency to increase the speed of switching systems leads to a situation in which the experimental study of switching processes is very difficult and in some cases, it is possible to obtain information only about the final result of the process, but not about its dynamics. In this case, only mathematical modeling is able to provide the required information. 
   Electrical breakdown of the gap between approaching contacts in vacuum circuit breakers occurs due to instantaneous explosion of a micro-asperity touching the contact surface. The arc igniting in ionized metallic vapors of a micro-asperity is the main source of failure in vacuum circuit breakers.
   Mathematical modeling of the dynamics of these phenomena is very important for understanding of arc evolution, mechanism of contact erosion and welding. Some models in this direction are presented in \cite{BiKhNo2002,BuBeVeZh1978,Ka1985,Kh1991,Kh1997,Kh2017,KhNoDa2003,KhNoAm2005,KhSaNo2016,KiOmKh1977,SiKs1991,Sl2014} . However, all thermophysical and electrical coefficients in the corresponding Stefan problems were taken constant. At the same time, for many composite materials used in modern electrical apparatus construction, these coefficients depend significantly on temperature. This work is an attempt to prove the unique solvability of such a Stefan problem with temperature-dependent coefficients in order to obtain a basis for developing specific methods for solving it.

Recently in Ref. \cite{NaBr2023} a generalized model  of a complete arc formation process from arc ignition to arc quenching was  presented. Temperature field for liquid
and solid phases in semi-infinite material where the cross-section effects are assumed negligible was modeled through a two phase Stefan type problem governed by generalized heat equations. The solution of the problem was obtained in terms of similarity variable, moreover, the temperature for two phases and free boundaries which describe the location of the boiling and melting interfaces, were determined. 

We consider the mathematical model with instantaneous explosion of a microasperity. This model can be applied for the total process of arcing from the arc ignition to its extinction. According to this model, a touching microasperity of a contact explodes instantaneous due to arc ignition with the power $P$ applied to the touching point $r=0$, which can be described by the $\delta$-function

\begin{equation}\label{1}
    P\delta(r,t)=P\cdot\dfrac{\exp\big(-\frac{r^2}{4a^2t}\big)}{2a\sqrt{\pi t}},
\end{equation}
where $a$ is the thermal diffusivity. Three domains $D_0$, $D_1$ and $D_2$ should be introduced for modeling the heat transfer. The region $D_0$  $(0<r<\alpha(t))$ is the zone of metal vapors, $D_1$ $(\alpha(t)<r<\beta(t))$ and $D_2$ $(\beta(t)<r<\infty)$ are the liquid and solid zones respectively  (see Figure \ref{fig1}).
\begin{figure}[h!!]
    \centering
    \includegraphics[scale=1]{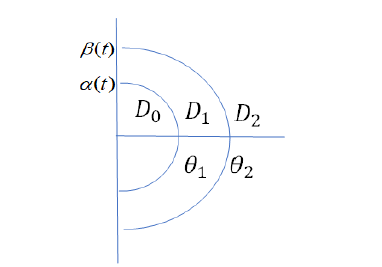}
  \caption{Spherical model of the electrical contact process. $D_0$ -- the sphere of metallic vapors, $D_1$ -- the sphere of liquid metal, $D_2$ -- the sphere of solid metal.}
   \label{fig1}
\end{figure}

The modeling of the temperature field in the vapor zone $D_0$  suggests that the temperature in this zone decreases from temperature $\theta_{ion}$ which is required for ionization of metallic vapors at the point of explosion $r=0$  to the boiling temperature $\theta_b$ on the boundary of vapor and liquid zone $r=\alpha(t)$. Taking into account that the thickness of the vapor zone is sufficiently small in comparison with the liquid zone, this zone $D_0$ can be considered as a heat resistance between arc and liquid zone, i.e. the temperature decreases linearly
\begin{equation}\label{2}
    \theta_0(r,t)=\theta_{ion}-(\theta_{ion}-\theta_b)\dfrac{r}{\alpha(t)}.
\end{equation}
The free boundary $r=\alpha(t)$ can be determined from the balance of the heat fluxes on this boundary 
\begin{equation}\label{3}
    \dfrac{P}{2a\sqrt{\pi t}}=-\lambda_b\dfrac{\partial\theta_0}{\partial r}\bigg|_{r=\alpha(t)}+L_b\gamma_b\alpha'(t).
\end{equation}

For the temperature of liquid and solid zones we have the spherical model equations with Joule heat source and heat flux condition, see \cite{Ho1981}
\begin{equation}\label{6}
    c_1(\theta_1)\gamma_1(\theta_1)\dfrac{\partial\theta_1}{\partial t}=\dfrac{1}{r^2}\dfrac{\partial}{\partial r}\bigg[\lambda_1(\theta_1)r^2\dfrac{\partial\theta_1}{\partial r}\bigg]+\rho_1(\theta_1)j_1^2,\;\;r\in D_1,\;\;t>0,
\end{equation}
\begin{equation}\label{7}
    c_2(\theta_2)\gamma_2(\theta_2)\dfrac{\partial\theta_2}{\partial t}=\dfrac{1}{r^2}\dfrac{\partial}{\partial r}\bigg[\lambda_2(\theta_2)r^2\dfrac{\partial\theta_2}{\partial r}\bigg]+\rho_2(\theta_2)j_2^2,\;\;r\in D_2,\;\;t>0,
\end{equation}
\begin{equation}\label{8}
{\color{black}    \lambda_1(\theta_1(\alpha(t),t))\dfrac{\partial\theta_1}{\partial t}\bigg|_{r=\alpha(t)}=-\dfrac{P\exp{-\big(\frac{r^2}{4a^2t}\big)}}{2a\sqrt{\pi t}}\bigg|_{r=\alpha(t)},\;\;t>0,}
\end{equation}
\begin{equation}\label{9}
    \theta_1(\beta(t),t)=\theta_m,\;\;\;\;\;t>0,
\end{equation}
\begin{equation}\label{10}
    \theta_2(\beta(t),t)=\theta_m,\;\;\;\;\;t>0,
\end{equation}
\begin{equation}\label{11}
    -\lambda_1(\theta_1(\beta(t),t))\dfrac{\partial\theta_1}{\partial r}\bigg|_{r=\beta(t)}=-\lambda_2(\theta_2(\beta(t),t))\dfrac{\partial\theta_2}{\partial r}\bigg|_{r=\beta(t)}+l_m\gamma_m\beta'(t),\;\;\;t>0,
\end{equation}
\begin{equation}\label{12}
    \theta_2(\infty,t)=0,\;\;\;\;\;t>0,
\end{equation}
\begin{equation}\label{13}
    \theta_2(r,0)=0,\;\;\alpha(0)=\beta(0)=0,
\end{equation}
where $c_i(\theta_i)$, $\gamma_i(\theta_i)$, $\lambda_i(\theta_i)$ and $\rho_i(\theta_i)j_i^2$, $i=1,2$ are the specific heat, the density, thermal conductivity and  the Joule heating, respectively, all of them dependent on temperature, $\theta_m$ is the  melting temperature, $l_m$ is a latent melting heat, $\gamma_m$ is a reference density at melting, $\alpha(t)$ is the free boundary at which power  of electricity is applied and $\beta(t)$ is the location of the liquid-solid interface which has to be found.
The purpose of this paper is to develop a mathematical model for the vaporization and melting of materials in closed electrical contacts. The aim is to understand and analyze the behavior of these materials under specific conditions, considering that the thermal coefficients depend on temperature and the influence of Joule heating. The paper focuses on solving the spherical heat equation and demonstrates that a solution can be obtained using similarity transformation, resulting in a nonlinear integral equation. In section 2 we transform the problem into a system of nonlinear integral equations. The steps involved in obtaining this system, emphasizing its relevance to the problem at hand. The significance of this transformation lies in simplifying the solution approach and facilitating further analysis. Section 3, discusses the existence and uniqueness of the solution to the system of integral equations obtained in the previous section. The paper presents mathematical arguments and proofs to support these claims, establishing the robustness and reliability of the solution method. The implications of this finding are emphasized, as it validates the feasibility of obtaining meaningful results from the model.

\section{Solution to the problem by using similarity principle}

The solution to the equation \eqref{3} is given by \begin{equation}\label{alpha}
\alpha(t)=2a\alpha_0\sqrt{t}
\end{equation}
where $\alpha_0$ solves the quadratic equation

\begin{equation}\label{4}
    \alpha_0^2-A\alpha_0+B=0
\end{equation}
where
\begin{equation}\label{5}
    A=\dfrac{P}{2a^2\sqrt{\pi}L_b\gamma_b},\;\;B=\dfrac{\lambda_b(\theta_{ion}-\theta_b)}{2a^2 L_b\gamma_b}.
\end{equation}
The equation \eqref{4} has two real roots if 
\begin{equation}\label{condP}
    P\geq2a\sqrt{2\pi}\sqrt{\lambda_b L_b\gamma_b(\theta_{ion}-\theta_b)},
\end{equation}
it means that the arc heat flux $P$ should be sufficiently large to provided the beginning of boiling. However the root $\alpha_{01}=\frac{2B}{A+\sqrt{A^{2}-4B}}$ defining the velocity of the boiling isotherm decreases when the heat flux $P$ (and also $A$) increases. This is in conflict with the physical meaning  of boiling process. On the contrary the root $\alpha_{02}=\frac{2B}{A-\sqrt{A^{2}-4B}}$ increases if the heat flux $P$ (i.e. $A$) increases, as it should be really. Therefore \eqref{4} has a unique solution given by $\alpha_{02}$ corresponding to the physical meaning at the condition \eqref{condP}.

We assume \eqref{condP} and $\alpha(t)$ given by \eqref{alpha} where \begin{equation}\label{alfacero}
    \alpha_0=\alpha_{02}=\frac{2B}{A-\sqrt{A^{2}-4B}}=\frac{2\sqrt{\pi}\lambda_b(\theta_{ion}-\theta_b)}{P-\sqrt{P^{2}-8a^{2}\pi L_b\gamma_b\lambda_b(\theta_{ion}-\theta_b)}}.
\end{equation}  By using dimensionless transformation
\begin{equation}\label{14}
    T_i(r,t)=\dfrac{\theta_i(r,t)-\theta_m}{\theta_m},\;\;(i=1,2)
\end{equation}
then problem \eqref{6}-\eqref{13} becomes
\begin{equation}\label{15}
    N_1(T_1)\dfrac{\partial T_1}{\partial t}=\dfrac{1}{r^2}\dfrac{\partial}{\partial r}\bigg[L_1(T_1)r^2\dfrac{\partial T_1}{\partial r}\bigg] + \frac{K_1(T_1)j_1^2}{\theta_m},\;\;r\in D_1,\;\; t>0,
\end{equation}
\begin{equation}\label{16}
   N_2(T_2)\dfrac{\partial T_2}{\partial t}=\dfrac{1}{r^2}\dfrac{\partial}{\partial r}\bigg[L_2(T_2)r^2\dfrac{\partial T_2}{\partial r}\bigg] + \frac{K_2(T_2)j_2^2}{\theta_m},\;\;r\in D_2,\;\; t>0,
\end{equation}
\begin{equation}\label{17}
    L_1(T_1(\alpha(t),t))\dfrac{\partial T_1}{\partial r}\bigg|_{r=\alpha(t)}=-\dfrac{P^*}{2a\sqrt{t}},\;\;\;t>0,
\end{equation}
\begin{equation}\label{18}
    T_1(\beta(t),t)=0,\;\;\;\;\;t>0,
\end{equation}
\begin{equation}\label{19}
    T_2(\beta(t),t)=0,\;\;\;\;\;t>0,
\end{equation}
\begin{equation}\label{20}
   {\color{blue} -L_1(T_1(\beta(t),t))\dfrac{\partial T_1}{\partial r}\bigg|_{r=\beta(t)}+L_2(T_2(\beta(t),t))\dfrac{\partial T_2}{\partial r}\bigg|_{r=\beta(t)}=\dfrac{l_m\gamma_m}{\theta_m}\beta'(t),\;\;t>0,}
\end{equation}
\begin{equation}\label{21}
    T_2(\infty,t)=-1,\;\;\;\;\;t>0,
\end{equation}
\begin{equation}\label{22}
    T_2(r,0)=-1,\;\;\;\alpha(0)=\beta(0)=0,
\end{equation}
where 
$$P^*=\dfrac{Pe^{-\alpha_0^2}}{\sqrt{\pi}\theta_m}, \quad N_i(T_i)=c_i(\theta_m(T_i(r,t)+1))\gamma_i(\theta_m(T_i(r,t)+1)),$$
$$L_i(T_i)=\lambda_i(\theta_m(T_i(r,t)+1)), \quad K_i(T_i)=\rho_i(\theta_m(T_i(r,t)+1)),\;\;\;i=1,2.$$
It is shown in {\color{blue} Ref. \cite{Ho1981}} that for the time arcing $t_a$ of the order $10^{-3}$ sec the sinusoidal current $I(t)=I_0\sin(\omega t)$  can be approximated by the expression
$I(t)=k\sqrt{t}$ with $ k=\tfrac{I_0\sin(\omega t_a)}{\sqrt{t_a}}$,
thus we have \begin{equation}\label{jotas}
 j_1=j_2=\dfrac{k\sqrt{t}}{2\pi r^2}
\end{equation} where $r$ is a radius of sphere. 

To solve the problem \eqref{15}-\eqref{22} we  propose the  similarity transformation
\begin{equation}\label{23}
    T_i(r,t)=u_i(\eta),\;\;\;\eta=\dfrac{r}{2a\sqrt{t}},\;\;i=1,2.
\end{equation}
From \eqref{19}, \eqref{22} and \eqref{23} the free boundary $\beta=\beta(t)$ can be represented as follows
\begin{equation}\label{24}
    \beta(t)=2a\xi\sqrt{t},
\end{equation}
where  $\xi$ is an unknown constant that has to be determined.

Then we obtain the following problem
\begin{equation}\label{25}
    [L_1^*(u_1)\eta^2u_1']'+2a^2\eta^3N_1^*(u_1)u_1'+\dfrac{k^2}{16a^2\pi^2\eta^2}K_1^*(u_1)=0,\;\;\;\alpha_0<\eta<\xi,
\end{equation}
\begin{equation}\label{26}
    [L_2^*(u_2)\eta^2u_2']'+2a^2\eta^3N_2^*(u_2)u_2'+\dfrac{k^2}{16a^2\pi^2\eta^2}K_2^*(u_2)=0,\;\;\;\xi<\eta,
\end{equation}
\begin{equation}\label{27}
    L_1^*(u_1(\alpha_0))u_1'(\alpha_0)=-P^*,
\end{equation}
\begin{equation}\label{28}
    u_1(\xi)=0,
\end{equation}
\begin{equation}\label{29}
    u_2(\xi)=0,
\end{equation}
\begin{equation}\label{30}
    L_1^*(u_1(\xi))u_1'(\xi)=L_2^*(u_2(\xi))u_2'(\xi)-M\xi,
\end{equation}
\begin{equation}\label{31}
    u_2(\infty)=-1,
\end{equation}
where $$M=\dfrac{2a^2l_m\gamma_m}{\theta_m},\quad  N_i^*(u_i)=N_i(u_i)=c_i(\theta_m(u_i+1))\gamma_i(\theta_m(u_i+1)), $$ $$ L_i^*(u_i)=L_i(u_i)=\lambda_i(\theta_m(u_i+1)),  \quad K_i^*(u_i)=\frac{K_i(u_i)}{\theta_{m}}={\rho_i(\theta_m(u_i+1))}{\theta_{m}},\;\;i=1,2.$$

By using the substitution
\begin{equation}\label{32}
    L_i^*(u_i)\eta^2 u_i'=\nu_i,\qquad i=1,2
\end{equation}
we conclude that $(\xi, u_1 ,u_2)$ is a solution to the problem \eqref{25}-\eqref{31} if and only if it  satisfies the following system of integral equations
\begin{equation}\label{33}
    u_1(\eta)=\alpha_0^2 P\Big(\chi_1(\xi,u_1,\alpha_0)-\chi_1(\eta,u_1,\alpha_0) \Big)+\Phi_1(\xi,u_1,\alpha_0)-\Phi_1(\eta,u_1,\alpha_0),\qquad \alpha_0\leq \eta\leq \xi
\end{equation}
\begin{equation}\label{34}
    u_2(\eta)=\Big(-1+\Phi_2(\infty,u_2,\xi)\Big)\frac{\chi_2(\eta,u_2,\xi)}{\chi_2(\infty,u_2,\xi)}-\Phi_2(\eta,u_2,\xi),\qquad \eta\geq \xi
\end{equation}
coupled with 
\begin{equation}\label{39}
    \frac{-1+\Phi_2(\infty,u_2,\xi)}{\chi_2(\infty,u_2,\xi)}+Q(\xi,u_1,\alpha_0)E_1(\xi,u_1,\alpha_0)=M\xi^3,\qquad \xi>\alpha_0
\end{equation}
where
\begin{equation}\label{35}
    \chi_1(\eta,u_1,\alpha_0)=\int\limits_{\alpha_0}^{\eta}\dfrac{E_1(v,u_1,\alpha_0)}{v^2L_1^*(u_1(v))}dv,\;\;\;\alpha_0\leq\eta\leq\xi,
\end{equation}
\begin{equation}\label{35a}
    \Phi_1(\eta,u_1,\alpha_0)=\dfrac{k^2}{16a^2\pi^2}\int\limits_{\alpha_0}^{\eta}\dfrac{E_1(v,u_1,\alpha_0)}{v^2L_1^*(u_1(v))}\int\limits_{\alpha_0}^{v}\dfrac{K_1^*(u_1(s))}{s^2E_1(s,u_1,\alpha_0)}dsdv,\;\;\;\alpha_0\leq\eta\leq\xi,
\end{equation}
\begin{equation}\label{37}
    E_1(\eta, u_1,\alpha_0)=\exp\Bigg(-2a^2\int\limits_{\alpha_0}^{\eta}v\dfrac{N_1^*(u_1(v))}{L_1^*(u_1(v))}dv\Bigg),\;\;\;\alpha_0\leq\eta\leq\xi,
\end{equation}
\begin{equation}\label{36}
    \chi_2(\eta,u_2,\xi)=\int\limits_{\xi}^{\eta}\dfrac{E_2(v,u_2,\xi)}{v^2L_2^*(u_2(v))}dv,\;\;\;\eta\geq\xi,
\end{equation}
\begin{equation}\label{36a}
    \Phi_2(\eta,u_2,\xi)=\dfrac{k^2}{16a^2\pi^2}\int\limits_{\xi}^{\eta}\dfrac{E_2(v,u_2,\xi)}{v^2L_2^*(u_2(v))}\int\limits_{\xi}^{v}\dfrac{K_2^*(u_2(s))}{s^2E_2(s,u_2,\xi)}dsdv,\;\;\;\eta\geq\xi,
\end{equation}

\begin{equation}\label{38}
    E_2(\eta, u_2,\xi)=\exp\Bigg(-2a^2\int\limits_{\xi}^{\eta}v\dfrac{N_2^*(u_2(v))}{L_2^*(u_2(v))}dv\Bigg),\;\;\;\eta\geq\xi,
\end{equation}
\begin{equation}\label{Q}
    Q(\xi,u_1,\alpha_0)=\alpha_0^2P^*+\dfrac{k^2 \mathcal{I}(\xi,u_1,\alpha_0)}{16a^2\pi^2},\qquad \mathcal{I}(\xi,u_1,\alpha_0)=\int\limits_{\alpha_0}^{\xi}\dfrac{K_1^*(u_1(s))}{s^2L_1^*(u_1(s))}ds, \qquad \xi>\alpha_0 .
\end{equation}


\section{Existence of solution to the problem}
In this section, we assumme hypothesis \eqref{alpha}, \eqref{condP}, \eqref{alfacero} and  \eqref{jotas} hold. For each $\xi>\alpha_0$ we will prove  the  existence and uniqueness of the solution to the integral equations \eqref{33} and \eqref{34}.

At first we consider the Banach space $(C^0[\alpha_0,\xi],||\cdot||)$  of continuous real valued functions endowed with supremum norm
$$||u_1||=\max_{\eta\in[\alpha_0,\xi]}|u_1(\eta)|$$
and the operator $V:C^0[\alpha_0,\xi]\to C^0[\alpha_0,\xi]$  defined by
\begin{equation}\label{40}
    V(u_1)(\eta)=\alpha_0^2 P\Big(\chi_1(\xi,u_1,\alpha_0)-\chi_1(\eta,u_1,\alpha_0) \Big)+\Phi_1(\xi,u_1,\alpha_0)-\Phi_1(\eta,u_1,\alpha_0).
\end{equation}

Through the Banach theorem, we will prove, for each $\xi>\alpha_0$, the existence and uniqueness of a fixed point of  the operator $V$  which will be the solution $u_1$ to equation \eqref{33}.

In order to solve equation \eqref{34} we consider the set $$\mathcal{K}=\{u_2\in C^0[\xi,\infty)/u_2(\xi)=0, u_2(\infty)=-1\}$$ which is a closed subset of the Banach space of the bounded continuous real valued functions defined on the interval $[\xi,\infty)$ with the supremum norm 
$$||u_2||=\sup_{\eta\in[\xi,\infty)}|u_2(\eta)|$$
Defining the operator $W:\mathcal{K}\to \mathcal{K}$ such that
\begin{equation}\label{41}
    W(u_2)(\eta)=\Big(-1+\Phi_2(\infty,u_2,\xi)\Big)\frac{\chi_2(\eta,u_2,\xi)}{\chi_2(\infty,u_2,\xi)}-\Phi_2(\eta,u_2,\xi),
\end{equation}
we will prove that, for each $\xi>\alpha_0$, there exist unique fixed point to the operator $W$ which is the unique solution  $u_2$ to the equation \eqref{34}.

We assume that $L^*,\;N^*$ and $K^*$ satisfy the following  inequalities
\begin{enumerate}
    \item[1)] There exists $L_m>0$ and $L_M>0$ such that
        \begin{equation}\label{44}
            L_m\leq L_1^*(u_1)\leq L_M,\;\;\forall u_1\in C^0[\alpha_0,\xi],
        \end{equation}
        \begin{equation}\label{45}
            L_m\leq L_2^*(u_2)\leq L_M,\;\;\forall u_2\in \mathcal{K}.
        \end{equation}
        There exists $\widetilde{L}_i>0,\;i=1,2$ such that
        \begin{equation}\label{46}
            ||L_1^*(u_1)-L_1^*(u_1^*)||\leq \widetilde{L}_1||u_1-u_1^*||,\;\;\forall u_1,u_1^*\in C^0[\alpha_0,\xi],
        \end{equation}
        \begin{equation}\label{47}
             ||L_2^*(u_2)-L_2^*(u_2^*)||\leq \widetilde{L}_2||u_2-u_2^*||,\;\;\forall u_2,u_2^*\in \mathcal{K},
        \end{equation}
    \item[2)] There exists $N_m>0$ and $N_M>0$ such that
        \begin{equation}\label{48}
            N_m\leq N_1^*(u_1)\leq N_M,\;\;\forall u_1\in C^0[\alpha_0, \xi],
        \end{equation}
        \begin{equation}\label{49}
            N_m\leq N_2^*(u_2)\leq N_M,\;\;\forall u_2\in \mathcal{K}.
        \end{equation}
        There exists $\widetilde{N}_i,\;i=1,2$  such that
        \begin{equation}\label{50}
            ||N_1^*(u_1)-N_1^*(u_1^*)||\leq \widetilde{N}_1||u_1-u_1^*||,\;\;\forall u_1,u_1^*\in C^0[\alpha_0, \xi],
        \end{equation}
        \begin{equation}\label{51}
            ||N_2^*(u_2)-N_2^*(u_2^*)||\leq \widetilde{N}_2||u_2-u_2^*||,\;\;\forall u_2,u_2^*\in \mathcal{K}.
        \end{equation}
    \item[3)] There exists $K_m>0$, $K_M>0$ and  $R\geq \frac{a^2 N_M}{L_m}$ such that
        \begin{equation}\label{52}
            K_m\leq K_1^*(u_1)\leq K_M,\;\;\forall u_1\in C^0[\alpha_0,\xi],
        \end{equation}
        \begin{equation}\label{53}
            K_2^*(u_2(s))\leq K_M \exp\left(-Rs^2 \right),\;\;\forall u_2\in \mathcal{K}.
        \end{equation}
        There exists $\widetilde{K}_i>0$ such that
        \begin{equation}\label{54}
            ||K_1^*(u_1)-K_1^*(u_1^*)||\leq \widetilde{K}_1||u_1-u_1^*||,\;\;\forall u_1,u_1^*\in C^0[\alpha_0, \xi],
        \end{equation}
        \begin{equation}\label{55}
            |K_2^*(u_2(s))-K_2^*(u_2^*(s))|\leq \widetilde{K}_2 \exp\left(-Rs^2 \right)||u_2-u_2^*||,\;\;\forall u_2,u_2^*\in \mathcal{K}.
        \end{equation}
\end{enumerate}

Now we will obtain some preliminary results to prove existence and uniqueness of  solution to equations \eqref{33} and \eqref{34}.

\begin{lemma}\label{lem1}
Assuming  \eqref{44}-\eqref{55}, then the following inequalities hold:
\begin{equation}\label{56}
    \exp\bigg(-a^2\dfrac{N_M}{L_m}(\eta^2-\alpha_0^2)\bigg)\leq E_1(\eta, u_1, \alpha_0)\leq \exp\bigg(-a^2\dfrac{N_m}{L_M}(\eta^2-\alpha_0^2)\bigg),\qquad \alpha_0\leq \eta\leq \xi
\end{equation}

\begin{equation}\label{57}
    \exp\bigg(-a^2\dfrac{N_M}{L_m}(\eta^2-\xi^2)\bigg)\leq E_2(\eta, u_2, \xi)\leq \exp\bigg(-a^2\dfrac{N_m}{L_M}(\eta^2-\xi^2)\bigg),\qquad \eta\geq \xi
\end{equation}

\begin{equation}\label{58}
    \begin{split}
  \chi_1(\eta,u_1,\alpha_0)\leq  \dfrac{  a\exp\bigg(a^2\dfrac{N_m}{L_M}\alpha_0^2\bigg)}{L_m} \sqrt{\dfrac{N_m}{L_M}} h(\eta,N_m,L_M,\alpha_0), \qquad \alpha_0\leq \eta\leq \xi
    \end{split}
\end{equation}

\begin{equation}\label{59}
    \begin{array}{ll}
     \chi_2(\eta,u_2,\xi)\geq \dfrac{  a\exp\bigg(a^2\dfrac{N_M}{L_m}\xi^2\bigg)}{L_M} \sqrt{\dfrac{N_M}{L_m}}h(\eta,N_M,L_m,\xi),\qquad \eta\geq \xi \\ \\
    \chi_2(\eta,u_2,\xi)\leq  \dfrac{  a\exp\bigg(a^2\dfrac{N_m}{L_M}\xi^2\bigg)}{L_m} \sqrt{\dfrac{N_m}{L_M}} h(\eta,N_m,L_M,\xi)\leq \dfrac{1}{L_m}\left(\dfrac{1}{\xi}-\dfrac{1}{\eta} \right) ,\qquad \eta\geq \xi.
    \end{array}
\end{equation}


\begin{equation}\label{61}
    \Phi_2(\eta,u_1,\xi)\leq \dfrac{k^2K_M}{16a^2\pi^{2}L_m}\dfrac{1}{\xi^2},\qquad \eta\geq \xi
\end{equation}
where  
\begin{equation}\label{h}
    \begin{array}{ll}
        h(\eta,N,L,z)&=\sqrt{\pi}\erf\bigg(a\sqrt{\tfrac{N}{L}}z\bigg)-\sqrt{\pi}\erf\bigg(a\sqrt{\tfrac{N}{L}}\eta\bigg)\\
        &+\tfrac{\sqrt{L}}{a\sqrt{N}z}\exp\bigg(-a^2\tfrac{N}{L}z^2\bigg)-\tfrac{\sqrt{L}}{a\sqrt{N}\eta}\exp\bigg(-a^2\tfrac{N}{L}\eta^2\bigg)
    \end{array}
\end{equation}

\end{lemma}
\begin{proof}
We are going to prove \eqref{56} and \eqref{58} by using definitions \eqref{35},   \eqref{37} and assumptions \eqref{44} and \eqref{48}. By using \eqref{37} we have
$$E_1(\eta,u_1,\alpha_0)\leq \exp\bigg(-2a^2\dfrac{N_m}{L_M}\int\limits_{\alpha_0}^{\eta}vdv\bigg)=\exp\bigg(-a^2\dfrac{N_m}{L_M}(\eta^2-\alpha_0^2)\bigg)$$
which gives us proof of \eqref{56} and with the help of definition \eqref{35} we get
$$\chi_1(\eta,u_1,\alpha_0)\leq \dfrac{\exp\bigg(a^2\dfrac{N_m}{L_M}\alpha_0^2\bigg)}{L_m}\int\limits_{\alpha_0}^{\eta}\dfrac{\exp\bigg(-a^2\frac{N_m}{L_M}v^2\bigg)}{ v^2}dv$$
by using the substitution $t=a\sqrt{\frac{N_m}{L_M}}v$ and the definition of Gauss error function $\erf(z)=\frac{2}{\sqrt{\pi}}\int\limits_0^z e^{-t^2}dt$ we obtain
$$\chi_1(\eta,u_1,\alpha_0)\leq \dfrac{a\exp\bigg(a^2\dfrac{N_m}{L_M}\alpha_0^2\bigg)}{L_m}\sqrt{\dfrac{N_m}{L_M}}\int\limits_{a\sqrt{N_m/L_M}\alpha_0}^{a\sqrt{N_m/L_M}\eta}\dfrac{e^{-t^2}}{ t^2}dt$$
$$\leq \dfrac{a\exp\bigg(a^2\dfrac{N_m}{L_M}\alpha_0^2\bigg)}{L_m}\sqrt{\dfrac{N_m}{L_M}}\Bigg[\dfrac{\sqrt{L_M}}{a\sqrt{N_m}\alpha_0}\exp\bigg(-a^2\dfrac{N_m}{L_M}\alpha_0^2\bigg)-\dfrac{\sqrt{L_M}}{a\sqrt{N_m}\eta}\exp\bigg(-a^2\dfrac{N_m}{L_M}\eta^2\bigg)$$
$$+\sqrt{\pi}\erf\bigg(a\sqrt{\dfrac{N_m}{L_M}}\alpha_0\bigg)-\sqrt{\pi}\erf\bigg(a\sqrt{\dfrac{N_m}{L_M}}\eta\bigg)\Bigg]\leq \dfrac{a\exp\bigg(a^2\dfrac{N_m}{L_M}\alpha_0^2\bigg)}{L_m}\sqrt{\dfrac{N_m}{L_M}}\;h(\eta,N_m,L_M,\alpha_0).$$

By using the definition \eqref{36}, \eqref{38} and using \eqref{45} and \eqref{49} we can prove \eqref{57} and \eqref{59}.

In order to prove  \eqref{61} we use \eqref{45} and \eqref{57}.  

Notice that for $\xi<s<v<\eta$ we have $E_2(v,u_2,\xi)\leq E_2(s,u_2,\xi)$, then
$$\Phi_2(\eta,u_2,\xi)\leq \dfrac{k^2 K_M}{16a^2\pi^2 L_m}\int\limits_{\xi}^{\eta}\dfrac{1}{v^2}\int\limits_{\xi}^{v}\dfrac{1}{s^2}dsdv\leq \dfrac{k^2 K_M}{16a^2\pi^2 L_m \xi}\int\limits_{\xi}^{\eta}\dfrac{1}{v^2} dv=\dfrac{k^2 K_M}{16a^2\pi^2 L_m \xi^2}  .$$
\end{proof}

\begin{lemma}\label{lem2}
If assumptions \eqref{44}-\eqref{55} hold then
\begin{enumerate}
    \item[1)] For all $\eta\in[\alpha_0,\xi]$ and $u_1,u_1^*\in C^0[\alpha_0,\xi]$ we have
        \begin{equation}\label{600}
            |E_1(\eta,u_1,\alpha_0)-E_1(\eta, u_1^*,\alpha_0)|\leq \widetilde{E}_1(\eta,\alpha_0)||u_1-u_1^*||,
        \end{equation}
        \begin{equation}\label{60a}
            |\chi_1(\eta,u_1,\alpha_0)-\chi_1(\eta,u_1^*,\alpha_0)|\leq \widetilde{\chi}_1(\eta,\alpha_0)||u_1-u_1^*||,
        \end{equation}
        where
\begin{equation}\label{tildeE1}
  \widetilde{E}_1(\eta, \alpha_0)=a^2\bigg(\dfrac{\widetilde{N}_1}{L_m}+\dfrac{N_M\widetilde{L}_1}{L_m^2}\bigg)(\eta^2-\alpha_0^2),  
\end{equation}
\begin{equation}\label{tildechi1}
    \widetilde{\chi}_1(\eta,\alpha_0)=\dfrac{a^2}{L_m}\bigg(\dfrac{\widetilde{N}_1}{L_m}+\dfrac{N_M\widetilde{L}_1}{L_m^2}\bigg)\bigg(\eta+\dfrac{\alpha_0^2}{\eta}-2\alpha_0\bigg)+\dfrac{\widetilde{L}_1}{L_m^2}\bigg(\dfrac{1}{\alpha_0}-\dfrac{1}{\eta}\bigg).
\end{equation}
    \item[2)] For all $\eta\in[\xi,\infty)$ and $u_2,u_2^*\in \mathcal{K}$ we have
        \begin{equation}\label{610}
            |E_2(\eta,u_2,\xi)-E_2(\eta, u_2^*,\xi)|\leq \widetilde{E}_2(\eta,\xi)||u_2-u_2^*||,
        \end{equation}
        \begin{equation}\label{61a}
            |\chi_2(\eta,u_2,\xi)-\chi_2(\eta,u_2^*,\xi)|\leq \widetilde{\chi}_2(\eta,\xi)||u_2-u_2^*||,
        \end{equation}
\end{enumerate}
where
\begin{equation}\label{tildeE2}
\widetilde{E}_2(\eta,\xi)=\exp\left(-a^2(\eta^2-\xi^2)\tfrac{N_m}{L_M} \right) \frac{a^2}{L_m^2}\big( L_M \widetilde{N}_2+N_M \widetilde{L}_2\big) (\eta^2-\xi^2)
\end{equation}

\[
    \widetilde{\chi}_2(\eta,\xi)= \tfrac{a L_M^{3/2} \sqrt{\pi}}{2L_m^4 \sqrt{N_m}}\big( L_M \widetilde{N}_2+N_M \widetilde{L}_2\big) \exp\left( a^2\xi^2\tfrac{N_m}{L_M}\right)
    \] 

\begin{equation}\label{tildechi2} \left(\erf\left(a\eta\sqrt{\tfrac{N_m}{L_M}} \right)-\erf\left(a\xi\sqrt{\tfrac{N_m}{L_M}} \right) \right) +\frac{\widetilde{L}_2}{L_m^2} \frac{1}{\xi}
\end{equation}

\end{lemma}
\begin{proof}
By using inequality $|\exp(-x)-\exp(-y)|<|x-y|$, definition \eqref{37} and assumptions \eqref{44}-\eqref{55} for $u_1,u_1^*\in C^0[\alpha_0,\xi]$  we obtain
$$|E_1(\eta,u_1,\alpha_0)-E_1(\eta,u_1^*,\alpha_0)|\leq 2a^2\int\limits_{\alpha_0}^{\eta}\Bigg|\dfrac{N_1^*(u_1)}{L_1^*(u_1)}-\dfrac{N_1^*(u_1^*)}{L_1^*(u_1^*)}\Bigg|s\; ds$$
$$\leq 2a^2\int\limits_{\alpha_0}^{\eta}\Bigg(\dfrac{|N_1^*(u_1)-N_1^*(u_1^*)|}{|L_1^*(u_1)|}+|N_1^*(u_1^*)|\dfrac{|L_1^*(u_1)-L_1^*(u_1^*)|}{|L_1^*(u_1)||L_1^*(u_1^*)|}\Bigg)s\; ds$$
$$\leq a^2\Bigg(\dfrac{\widetilde{N}_1}{L_m}+\dfrac{N_M\widetilde{L}_1}{L_m^2}\Bigg)(\eta^2-\alpha_0^2)||u_1-u_1^*||$$
and
$$|\chi_1(\eta,u_1,\alpha_0)-\chi_1(\eta,u_1^*,\alpha_0)|\leq \int\limits_{\alpha_0}^{\eta}\dfrac{\big|E_1(v,u_1,\alpha_0)L_1^*(u_1^*)-E_1(v,u_1^*,\alpha_0)L_1^*(u_1)\big|}{|L_1^*(u_1)||L_1^*(u_1^*)|}\dfrac{dv}{v^2}$$
$$\leq\int\limits_{\alpha_0}^{\eta}\dfrac{\big|E_1(v,u_1,\alpha_0)-E_1(v,u_1^*,\alpha_0)\big|}{|L_1^*(u_1)|}\dfrac{dv}{v^2}+\int\limits_{\alpha_0}^{\eta}\dfrac{|L_1^*(u_1^*)-L_1^*(u_1)|}{|L_1^*(u_1)||L_1^*(u_1^*)|}\big|E_1(v,u_1^*,\alpha_0)\big|\dfrac{dv}{v^2}:= G_1+G_2$$
where
$$G_1\leq \dfrac{a^2}{L_m}\bigg(\dfrac{\widetilde{N}_1}{L_m}+\dfrac{N_M\widetilde{L}_1}{L_m^2}\bigg)||u_1-u_1^*||\int\limits_{\alpha_0}^{\eta}\dfrac{v^2-\alpha_0^2}{v^2}dv$$
$$\leq \dfrac{a^2}{L_m}\bigg(\dfrac{\widetilde{N}_1}{L_m}+\dfrac{N_M\widetilde{L}_1}{L_m^2}\bigg)\bigg(\eta+\dfrac{\alpha_0^2}{\eta}-2\alpha_0\bigg)||u_1-u_1^*||,$$
$$G_2\leq \dfrac{\widetilde{L}_1}{L_m^2}||u_1-u_1^*||\int\limits_{\alpha_0}^{\eta}\dfrac{dv}{v^2}\leq \dfrac{\widetilde{L}_1}{L_m^2}\bigg(\dfrac{1}{\alpha_0}-\dfrac{1}{\eta}\bigg)||u_1-u_1^*||,$$
then it implies that
$G_1+G_2=\widetilde{\chi}_1(\eta,\alpha_0)||u_1-u_1^*||.$

We prove \eqref{610} in a similar way to \eqref{600}:
$$\big|E_2(\eta,u_2,\xi)-E_2(\eta,u_2^*,\xi)\big|\leq \exp\left(-a^2(\eta^2-\xi^2)\tfrac{N_m}{L_M} \right) \frac{a^2}{L_m^2}\big( L_M \widetilde{N}_2+N_M \widetilde{L}_2\big) (\eta^2-\xi^2)||u_2-u_2^*||.$$
For inequality \eqref{61a}, notice that
$$\begin{array}{ll}
&|\chi_2(\eta,u_2,\xi)-\chi_2(\eta,u_2^*,\xi)|\leq \displaystyle\int\limits_{\xi}^{\eta}  \dfrac{|E_2(v,u_2,\xi)-E_2(v,u_2^*,\xi)| L_M+\widetilde{L}_2 ||u_2-u_2^*||}{L_m^2}\dfrac{1}{v^2}  dv\\
&\leq  \left[ \frac{a^2 L_M}{L_m^4}\big( L_M \widetilde{N}_2+N_M \widetilde{L}_2\big) \displaystyle\int\limits_{\xi}^{\eta} \exp\left(-a^2(v^2-\xi^2)\tfrac{N_m}{L_M} \right)   dv+ \frac{\widetilde{L}_2}{L_m^2} \frac{1}{\xi}\right]  ||u_2-u_2^*||\\
& \leq \left[ \tfrac{a L_M^{3/2} \sqrt{\pi}}{2L_m^4 \sqrt{N_m}}\big( L_M \widetilde{N}_2+N_M \widetilde{L}_2\big) \exp\left( a^2\xi^2\tfrac{N_m}{L_M}\right) \Bigg(\erf\left(a\eta\sqrt{\tfrac{N_m}{L_M}} \right)-\erf\left(a\xi\sqrt{\tfrac{N_m}{L_M}} \right) \Bigg)\right.\\ \\
&\left. + \frac{\widetilde{L}_2}{L_m^2} \frac{1}{\xi}\right]  ||u_2-u_2^*||
\end{array}$$

\end{proof}

\begin{lemma}\label{lem3}
We suppose that  \eqref{44}-\eqref{55} hold then
\begin{enumerate}
    \item[1)] For all $\eta\in[\alpha_0,\xi]$ and $u_1,u_1^*\in C^0[\alpha_0,\xi]$ we have
        \begin{equation}\label{62}
            |\Phi_1(\eta,u_1,\alpha_0)-\Phi_1(\eta,u_1^*,\alpha_0)|\leq\widetilde{\Phi}_1(\eta,\alpha_0)||u_1-u_1^*||, 
        \end{equation}
        where
\begin{equation}\label{tildePhi}
\widetilde{\Phi}_1(\eta,\alpha_0)=\tfrac{k^2}{16a^2\pi^2}\tfrac{K_M}{L_m}a^2\bigg(\tfrac{\widetilde{N_1}}{L_m}+\tfrac{N_M\widetilde{L_1}}{L_m^2}\bigg)\exp\bigg(a^2\tfrac{N_m}{L_M}(\eta^2-\alpha_0^2)\bigg)\Bigg[\ln\bigg|\tfrac{\alpha_0}{\eta}\bigg|+\tfrac{\eta}{\alpha_0}+\tfrac{\alpha_0}{\eta}-\tfrac{\alpha_0^2}{2\eta^2}-\tfrac{3}{2}\Bigg]
       \end{equation}
      
\[   +\tfrac{k^2}{16a^2\pi^2}\big(\dfrac{K_M}{L_m^2}+\dfrac{\widetilde{K_1}}{L_m}\widetilde{L_1}\big)\exp\bigg(a^2\dfrac{N_M}{L_m}(\eta^2-\alpha_0^2)\bigg)\bigg[\dfrac{1}{2\alpha_0^2}+\dfrac{1}{2\eta^2}-\dfrac{1}{\alpha_0\eta}\bigg]
        \]
         \[    + \dfrac{k^2}{16a^2\pi^2}\dfrac{K_M}{L_m}a^2\bigg(\dfrac{\widetilde{N_1}}{L_m}+\dfrac{N_M\widetilde{L_1}}{L_m^2}\bigg)\exp^2\bigg(a^2\dfrac{N_M}{L_m}(\eta^2-\alpha_0^2)\bigg)\bigg[\ln\left| \dfrac{\eta}{\alpha_0}\right|-\dfrac{\alpha_0^2}{2\eta^2}-\dfrac{3}{2}+\dfrac{2\alpha_0}{\eta} \bigg] 
 \]
       \\
       
    \item[2)] For all $\eta\in[\xi,\infty)$ and $u_2,u_2^*\in K$ we have
        \begin{equation}\label{63}
            |\Phi_2(\eta,u_2,\xi)-\Phi_2(\eta,u_2^*,\xi)|\leq\widetilde{\Phi}_2(\eta,\xi)||u_2-u_2^*||, 
        \end{equation}
        where

       \begin{equation}\label{phi2tilde}
\widetilde{\Phi}_2(\eta,\xi)=\dfrac{k^2K_M}{16\pi^{3/2}L_m^3}  \frac{\big( L_M \widetilde{N}_2+N_M \widetilde{L}_2\big)}{\sqrt{R-a^2\tfrac{N_M}{L_m}}} \dfrac{1}{\xi}+\dfrac{k^2\widetilde{K}_M}{16a^2\pi^{3/2} L_m  \sqrt{R-a^2\tfrac{N_M}{L_m}}}\, \dfrac{1}{\xi}
\end{equation}
\[+\dfrac{k^2K_M}{16a^2\pi^2}\left(\dfrac{a^2L_M}{L_m^4\xi} \big( L_M \widetilde{N}_2+N_M \widetilde{L}_2\big)\frac{\sqrt{\pi}}{\sqrt{R-a^2\tfrac{N_M}{L_m}}}+\dfrac{\widetilde{L}_2}{L_m^2} \dfrac{1}{\xi^2}\right)
       \]
\end{enumerate}

\medskip

\end{lemma}

\begin{proof}
Suppose \eqref{44}-\eqref{55} hold then by using Lemma \ref{lem1}, \ref{lem2} and definition \eqref{35a} we have
$$|\Phi_1(\eta,u_1,\alpha_0)-\Phi_1(\eta,u_1^*,\alpha_0)|$$
$$\leq\dfrac{k^2}{16a^2\pi^2}\int\limits_{\alpha_0}^{\eta}\Bigg|\dfrac{E_1(v,u_1,\alpha_0)}{L_1^*(u_1)}\int\limits_{\alpha_0}^v\dfrac{K_1^*(u_1)ds}{s^2E_1(s,u_1,\alpha_0)}-\dfrac{E_1(v,u_1^*,\alpha_0)}{L_1^*(u_1^*)}\int\limits_{\alpha_0}^v\dfrac{K_1^*(u_1^*)ds}{s^2E_1(s,u_1^*,\alpha_0)}\Bigg|\dfrac{dv}{v^2}$$

$$\leq\dfrac{k^2}{16a^2\pi^2}\int\limits_{\alpha_0}^{\eta}\dfrac{\bigg|E_1(v,u_1,\alpha_0)L_1^*(u_1^*)\int\limits_{\alpha_0}^v\frac{K_1^*(u_1)ds}{s^2E_1(s,u_1,\alpha_0)}-E_1(v,u_1^*,\alpha_0)L_1^*(u_1^*)\int\limits_{\alpha_0}^v\frac{K_1^*(u_1)ds}{s^2E_1(s,u_1,\alpha_0)}\bigg|}{|L_1^*(u_1)|\cdot|L_1^*(u_1^*)|}\dfrac{dv}{v^2}$$
$$+\dfrac{k^2}{16a^2\pi^2}\int\limits_{\alpha_0}^{\eta}\dfrac{\bigg|E_1(v,u_1^*,\alpha_0)L_1^*(u_1^*)\int\limits_{\alpha_0}^v\frac{K_1^*(u_1)ds}{s^2E_1(s,u_1,\alpha_0)}-E_1(v,u_1^*,\alpha_0)L_1^*(u_1)\int\limits_{\alpha_0}^v\frac{K_1^*(u_1^*)ds}{s^2E_1(s,u_1^*,\alpha_0)}\bigg|}{|L_1^*(u_1)|\cdot|L_1^*(u_1^*)|}\dfrac{dv}{v^2}$$
$$:=J_1+J_2$$
where
$$J_1\leq  \dfrac{k^2}{16a^2\pi^2}\int\limits_{\alpha_0}^{\eta}\dfrac{|E_1(v,u_1,\alpha_0)-E_1(v,u_1^*,\alpha_0)|}{|L_1^*(u_1)|}\Bigg|\int\limits_{\alpha_0}^{v}\dfrac{K_1^*(u_1)ds}{s^2E_1(s,u_1,\alpha_0)}\Bigg|\dfrac{dv}{v^2}$$
$$\leq\dfrac{k^2}{16a^2\pi^2}\dfrac{K_M}{L_m}a^2\bigg(\dfrac{\widetilde{N}_1}{L_m}+\dfrac{N_M\widetilde{L}_1}{L_m^2}\bigg)\exp\bigg(a^2\dfrac{N_m}{L_M}(\eta^2-\alpha_0^2)\bigg)||u_1-u_1^*||\int\limits_{\alpha_0}^{\eta}\dfrac{v^2-\alpha_0^2}{v^2}\int\limits_{\alpha_0}^v\dfrac{ds}{s^2}dv$$
$$\leq \dfrac{k^2}{16a^2\pi^2}\dfrac{K_M}{L_m}a^2\bigg(\dfrac{\widetilde{N}_1}{L_m}+\dfrac{N_M\widetilde{L}_1}{L_m^2}\bigg)\exp\bigg(a^2\dfrac{N_m}{L_M}(\eta^2-\alpha_0^2)\bigg)\Bigg[\ln\bigg|\dfrac{\alpha_0}{\eta}\bigg|+\dfrac{\eta}{\alpha_0}+\dfrac{\alpha_0}{\eta}-\dfrac{\alpha_0^2}{2\eta^2}-\dfrac{3}{2}\Bigg]||u_1-u_1^*||.$$
and 
$$J_2\leq \dfrac{k^2}{16a^2\pi^2}\int\limits_{\alpha_0}^{\eta}\dfrac{\bigg|L_1^*(u_1^*)\int_{\alpha_0}^{v}\frac{K_1^*(u_1)ds}{s^2E_1(s,u_1,\alpha_0)}-L_1^*(u_1)\int_{\alpha_0}^{v}\frac{K_1^*(u_1^*)ds}{s^2E_1(s,u_1^*,\alpha_0)}\bigg|}{|L_1^*(u_1)|\cdot |L_1^*(u_1^*)|}\big|E_1(v,u_1^*,\alpha_0)\big|\dfrac{dv}{v^2}$$
$$\leq \dfrac{k^2}{16a^2\pi^2}\int\limits_{\alpha_0}^{\eta}\dfrac{\bigg|L_1^*(u_1^*)\int_{\alpha_0}^{v}\frac{K_1^*(u_1)ds}{s^2E_1(s,u_1,\alpha_0)}-L_1^*(u_1)\int_{\alpha_0}^{v}\frac{K_1^*(u_1)ds}{s^2E_1(s,u_1,\alpha_0)}\bigg|}{|L_1^*(u_1)|\cdot |L_1^*(u_1^*)|}\big|E_1(v,u_1^*,\alpha_0)\big|\dfrac{dv}{v^2}$$
$$+\dfrac{k^2}{16a^2\pi^2}\int\limits_{\alpha_0}^{\eta}\dfrac{\bigg|L_1^*(u_1)\int_{\alpha_0}^{v}\frac{K_1^*(u_1)ds}{s^2E_1(s,u_1,\alpha_0)}-L_1^*(u_1)\int_{\alpha_0}^{v}\frac{K_1^*(u_1^*)ds}{s^2E_1(s,u_1^*,\alpha_0)}\bigg|}{|L_1^*(u_1)|\cdot |L_1^*(u_1^*)|}\big|E_1(v,u_1^*,\alpha_0)\big|\dfrac{dv}{v^2}$$
$$:= H_1+H_2$$
where
$$H_1\leq  \dfrac{k^2}{16a^2\pi^2}\int\limits_{\alpha_0}^{\eta}\dfrac{|L_1^*(u_1^*)-L_1^*(u_1)|}{|L_1^*(u_1)|\cdot|L_1^*(u_1^*)|}\Bigg|\int\limits_{\alpha_0}^{v}\dfrac{K_1^*(u_1)ds}{s^2E_1(s,u_1,\alpha_0)}\Bigg|\dfrac{dv}{v^2}$$
$$\leq\dfrac{k^2}{16a^2\pi^2}\dfrac{K_M}{L_m^2}\widetilde{L_1}\exp\bigg(a^2\dfrac{N_M}{L_m}(\eta^2-\alpha_0^2)\bigg)||u_1-u_1^*||\int\limits_{\alpha_0}^{\eta}\dfrac{1}{v^2}\int\limits_{\alpha_0}^{v}\dfrac{1}{s^2}dsdv$$
$$\leq\dfrac{k^2}{16a^2\pi^2}\dfrac{K_M}{L_m^2}\widetilde{L_1}\exp\bigg(a^2\dfrac{N_M}{L_m}(\eta^2-\alpha_0^2)\bigg)\bigg[\dfrac{1}{2\alpha_0^2}+\dfrac{1}{2\eta^2}-\dfrac{1}{\alpha_0\eta}\bigg]||u_1-u_1^*||,$$
$$H_2\leq \dfrac{k^2}{16a^2\pi^2}\int\limits_{\alpha_0}^{\eta}\dfrac{\bigg|\int_{\alpha_0}^{v}\frac{K_1^*(u_1)ds}{s^2E_1(s,u_1,\alpha_0)}-\int\limits_{\alpha_0}^{v}\frac{K_1^*(u_1^*)ds}{s^2E_1(s,u_1^*,\alpha_0)}\bigg|}{|L_1^*(u_1^*)|}\dfrac{dv}{v^2}$$

$$\leq\dfrac{k^2}{16a^2\pi^2}\dfrac{1}{L_m}\int\limits_{\alpha_0}^{\eta}\dfrac{1}{v^2}\int\limits_{\alpha_0}^{v}\dfrac{\big|K_1^*(u_1)E_1(s,u_1^*,\alpha_0)-K_1^*(u_1^*)E_1(s,u_1^*,\alpha_0)\big|}{\big|E_1(s,u_1,\alpha_0)\big|\big|E_1(s,u_1^*,\alpha_0)\big|}\dfrac{1}{s^2}dsdv$$
$$+\dfrac{k^2}{16a^2\pi^2}\dfrac{1}{L_m}\int\limits_{\alpha_0}^{\eta}\dfrac{1}{v^2}\int\limits_{\alpha_0}^{v}\dfrac{\big|K_1^*(u_1^*)E_1(s,u_1^*,\alpha_0)-K_1^*(u_1^*)E_1(s,u_1,\alpha_0)\big|}{\big|E_1(s,u_1,\alpha_0)\big|\big|E_1(s,u_1^*,\alpha_0)\big|}\dfrac{1}{s^2}dsdv$$
$$:= C_1+C_2$$
where
$$C_1\leq \dfrac{k^2}{16a^2\pi^2}\dfrac{1}{L_m}\int\limits_{\alpha_0}^{\eta}\dfrac{1}{v^2}\int\limits_{\alpha_0}^{v}\frac{\big|K_1^*(u_1)-K_1^*(u_1^*)\big|}{\big|E_1(s,u_1,\alpha_0)\big|}\dfrac{1}{s^2}dsdv$$
$$\leq \dfrac{k^2}{16a^2\pi^2}\dfrac{\widetilde{K}_1}{L_m}\exp\bigg(a^2\dfrac{N_M}{L_m}(\eta^2-\alpha_0^2)\bigg)||u_1-u_1^*||\int\limits_{\alpha_0}^{\eta}\dfrac{1}{v^2}\int\limits_{\alpha_0}^{v}\dfrac{1}{s^2}dsdv$$
$$\leq\dfrac{k^2}{16a^2\pi^2}\dfrac{\widetilde{K}_1}{L_m}\exp\bigg(a^2\dfrac{N_M}{L_m}(\eta^2-\alpha_0^2)\bigg)\bigg[\dfrac{1}{2\alpha_0^2}+\dfrac{1}{2\eta^2}-\dfrac{1}{\alpha_0\eta}\bigg]||u_1-u_1^*||,$$
$$C_2\leq \dfrac{k^2}{16a^2\pi^2}\dfrac{1}{L_m}\int\limits_{\alpha_0}^{\eta}\dfrac{1}{v^2}\int\limits_{\alpha_0}^{v}\dfrac{\big|E_1(s,u_1^*,\alpha_0)-E_1(s,u_1,\alpha_0)\big|}{\big|E_1(s,u_1^*,\alpha_0)\big| \big|E_1(s,u_1,\alpha_0)\big|}\big|K_1^*(u_1^*)\big|\dfrac{1}{s^2}dsdv$$
$$\leq\dfrac{k^2}{16a^2\pi^2}\dfrac{K_M}{L_m}a^2\bigg(\dfrac{\widetilde{N}_1}{L_m}+\dfrac{N_M\widetilde{L}_1}{L_m^2}\bigg)\exp^2\bigg(a^2\dfrac{N_M}{L_m}(\eta^2-\alpha_0^2)\bigg)||u_1-u_1^*||\int\limits_{\alpha_0}^{\eta}\dfrac{1}{v^2}\int\limits_{\alpha_0}^{v}\dfrac{s^2-\alpha_0^2}{s^2}dsdv$$
$$\leq\dfrac{k^2}{16a^2\pi^2}\dfrac{K_M}{L_m}a^2\bigg(\dfrac{\widetilde{N}_1}{L_m}+\dfrac{N_M\widetilde{L}_1}{L_m^2}\bigg)\exp^2\bigg(a^2\dfrac{N_M}{L_m}(\eta^2-\alpha_0^2)\bigg)\bigg[\ln\left| \dfrac{\eta}{\alpha_0}\right|-\dfrac{\alpha_0^2}{2\eta^2}-\dfrac{3}{2}+\dfrac{2\alpha_0}{\eta} \bigg]||u_1-u_1^*||.$$

Then we have 
$$H_2\leq \left[\dfrac{k^2}{16a^2\pi^2}\dfrac{\widetilde{K_1}}{L_m}\exp\bigg(a^2\dfrac{N_M}{L_m}(\eta^2-\alpha_0^2)\bigg)\bigg[\dfrac{1}{2\alpha_0^2}+\dfrac{1}{2\eta^2}-\dfrac{1}{\alpha_0\eta}\bigg] \right.$$
$$\left. +\dfrac{k^2}{16a^2\pi^2}\dfrac{K_M}{L_m}a^2\bigg(\dfrac{\widetilde{N}_1}{L_m}+\dfrac{N_M\widetilde{L}_1}{L_m^2}\bigg)\exp^2\bigg(a^2\dfrac{N_M}{L_m}(\eta^2-\alpha_0^2)\bigg)\bigg[\ln\left| \dfrac{\eta}{\alpha_0}\right|-\dfrac{\alpha_0^2}{2\eta^2}-\dfrac{3}{2}+\dfrac{2\alpha_0}{\eta} \bigg]    \right]||u_1-u_1^*||$$
and finally we obtain
$$J_1+J_2\leq\widetilde{\Phi}_1(\alpha_0,\eta)||u_1-u_1^*||.$$
The inequality \eqref{63} can be proved in a similar way.

$$\bigg| \Phi_2(\eta,u_2,\xi)-\Phi_2(\eta,u_2^*,\xi)\bigg|\leq  \dfrac{k^2}{16a^2\pi^2} \left[\displaystyle\int\limits_{\xi}^\eta \frac{E_2(v,u_2,\xi)}{v^2 L_m} \displaystyle\int\limits_{\xi}^v \frac{1}{s^2} \frac{K_2^*(u_2)|E_2(v,u_2,\xi)-E_2(s,u_2^*,\xi)|}{E_2(s,u_2,\xi)E_2(s,u_2^*,\xi)} ds dv\right.$$
$$\left.+ \displaystyle\int\limits_{\xi}^\eta \frac{E_2(v,u_2^*,\xi)}{v^2 L_m} \displaystyle\int\limits_{\xi}^v \frac{1}{s^2} \frac{|K_2^*(u_2)-K_2^*(u_2^*)|}{E_2(s,u_2^*,\xi)} ds dv\right.$$ 
$$\left.+ \displaystyle\int\limits_{\xi}^\eta \dfrac{1}{v^2} \left|\frac{E_2(v,u_2,\xi)}{L_2^*(u_2)}-\frac{E_2(v,u_2^*,\xi)}{L_2^*(u_2^*)}\right|  \displaystyle\int\limits_{\xi}^v \dfrac{K_2^*(u_2^*)}{s^2 E_2(s,u_2^*,\xi)} ds dv\right] := M_1+M_2+M_3$$

On one hand, 
$$
    M_1\leq  \dfrac{k^2}{16a^2\pi^2} \displaystyle\int\limits_{\xi}^\eta \frac{1}{v^2 L_m} \displaystyle\int\limits_{\xi}^v \frac{1}{s^2} \frac{K_2^*(u_2)|E_2(v,u_2,\xi)-E_2(s,u_2^*,\xi)|}{E_2(s,u_2^*,\xi)} ds dv 
$$
$$\leq  \dfrac{k^2K_M}{16\pi^2L_m^3} \big( L_M \widetilde{N}_2+N_M \widetilde{L}_2\big) \displaystyle\int\limits_{\xi}^\eta \frac{1}{v^2} \displaystyle\int\limits_{\xi}^v \exp\left(-s^2\left(R-\tfrac{a^2N_M}{L_m} \right) \right)     ds dv  ||u_2-u_2^*||
$$
 $$\leq  \dfrac{k^2K_M}{16\pi^{3/2}L_m^3}  \frac{\big( L_M \widetilde{N}_2+N_M \widetilde{L}_2\big)}{\sqrt{R-a^2\tfrac{N_M}{L_m}}} \dfrac{1}{\xi} ||u_2-u_2^*||. $$
On the other hand,
$$
    M_2\leq \dfrac{k^2\widetilde{K}_2}{16a^2\pi^2 L_m} \displaystyle\int\limits_{\xi}^\eta \frac{1}{v^2} \displaystyle\int\limits_{\xi}^v \exp\left(-s^2\left(R-\tfrac{a^2N_M}{L_m} \right) \right)     ds dv 
  $$  $$\leq \frac{k^2\widetilde{K}_2}{16a^2\pi^{3/2} L_m  \sqrt{R-a^2\tfrac{N_M}{L_m}}}\, \dfrac{1}{\xi}\, ||u_2-u_2^*||.
$$

Taking into account that
$$\begin{array}{ll}
\left|\dfrac{E_2(v,u_2,\xi)}{L_2^*(u_2)}-\dfrac{E_2(v,u_2^*,\xi)}{L_2^*(u_2^*)}\right| \\ \\ \leq 
\dfrac{|E_2(v,u_2,\xi)-E_2(v,u_2^*,\xi)| L_2^*(u_2)+|L_2^*(u_2)-L_2^*(u_2^*)| E_2(v.u_2^*,\xi)}{L_2^*(u_2) L_2^*(u_2^*)}\\
\\ \leq \left(\dfrac{a^2L_M}{L_m^4} \big( L_M \widetilde{N}_2+N_M \widetilde{L}_2\big) v^2+\dfrac{\widetilde{L}_2}{L_m^2}\right)||u_2-u_2^*||
\end{array}
$$
and
$$\displaystyle\int\limits_{\xi}^v \dfrac{K_2^*(u_2^*)}{s^2 E_2(s,u_2^*,\xi)} ds \leq \dfrac{K_M}{\xi} \exp\left( -\left( R-a^2\tfrac{N_M}{L_m}\right) v^2\right),$$
then
\begin{equation}
M_3\leq \dfrac{k^2K_M}{16a^2\pi^2}\left(\dfrac{a^2L_M}{L_m^4\xi} \big( L_M \widetilde{N}_2+N_M \widetilde{L}_2\big)\frac{\sqrt{\pi}}{\sqrt{R-a^2\tfrac{N_M}{L_m}}}+\dfrac{\widetilde{L}_2}{L_m^2} \dfrac{1}{\xi^2}\right)||u_2-u_2^*||.
\end{equation}

Therefore
$$ M_1+M_2+M_3 \leq \widetilde{\Phi}_2(\eta,\xi)||u_2-u_2^*||.$$

\end{proof}

Now we will obtain the main result of this work from the following theorem.

\begin{lemma}\label{th1}
Suppose \eqref{44}-\eqref{55} hold. Then, for each $\alpha_0<\xi<\Bar{\xi}_1$, the 
self-map operator $V: C^0[\alpha_0,\xi]\to C^0[\alpha_0,\xi]$ defined by \eqref{40} is a contraction operator, 
 where $\Bar{\xi}_1>0$ is defined as the  unique solution to $\varepsilon_1(z,\alpha_0)=1$ with
\begin{equation}\label{64}
    \varepsilon_1(z):=2\alpha_0^2 P\widetilde{\chi}_1(z,\alpha_0)+\widetilde{\Phi}_1(z,\alpha_0).
\end{equation}
\end{lemma}

\begin{proof}

Let $u_1,\;u_1^*\in C^0[\alpha_0,\xi]$. 

For $\eta\in [\alpha_0,\xi] $, we have
$$\begin{array}{ll}
|V(u_1)(\eta)-V(u_1^*)(\eta)|\leq 2\alpha_0^2P \Big|\chi_1(\xi,u_1,\alpha_0)-\chi_1(\xi,u_1^*,\alpha_0)\Big|
\\ \\ +2\Big|\Phi_1(\xi,u_1,\alpha_0)-\Phi_1(\xi,u_1^*,\alpha_0)\Big|
\leq \bigg(2\alpha_0^2 P\widetilde{\chi}_1(\xi,\alpha_0)+\widetilde{\Phi}_1(\xi,\alpha_0) \bigg)||u_1-u_1^*||
\end{array}$$
then 
$$||V(u_1)-V(u_1^*)||\leq \varepsilon_1(\xi)||u_1-u_1^*|| $$
where the function $\varepsilon_1(z)$ is defined by \eqref{64}. 

We can easily check that function $\varepsilon_1$ is an increasing function which goes from $0$ to $+\infty$ when $z$ goes from $\alpha_0$ to $+\infty$.

Therefore, there exists a unique $\Bar{\xi}_1>0$ such that $\varepsilon_1(\Bar{\xi}_1)=1$ and 
$$\varepsilon_1(z)<1,\;\; \forall \alpha_0\leq z<\bar{\xi}_1,\;\;\text{and}\;\;\varepsilon_1(z)>1,\;\;\forall z>\bar{\xi}_1.$$

If we take $\xi$ such that $\alpha_0<\xi<\Bar{\xi}_1$, it implies $\varepsilon_1(\xi)<1$ then we can conclude that operator $V$ defined by \eqref{40} is a contraction operator. 

\end{proof}

\begin{theorem}\label{th1-a}
Suppose \eqref{44}-\eqref{55} hold. Then, for each $\alpha_0<\xi<\Bar{\xi}_1$ there exists a unique solution $u_1\in C^0[\alpha_0,\xi]$ to integral equation \eqref{33}.
\end{theorem}

\begin{proof}
It follows immediately from Lemma \ref{th1}.
\end{proof}

\begin{lemma}\label{th2}
Suppose \eqref{44}-\eqref{55} hold and assume that the function
\begin{equation}\label{epsilon2}
    \begin{array}{ll}
\varepsilon_2(z):= 2\widetilde{\Phi}_2(+\infty,\alpha_0)+\dfrac{2L_M\left[\mathcal{G}\left( az \tfrac{\sqrt{N_m}}{\sqrt{L_M}}\right) \frac{L_M^2}{2L_m^2N_m} +\widetilde{L}_2\right]}{L_m^2 \left[ 1-\mathcal{G}\left( az \tfrac{\sqrt{N_M}}{\sqrt{L_m}}\right)\right]}\left(1+
\dfrac{k^2K_M}{16a^2\pi^{2}L_m}\dfrac{1}{\alpha_0^2} \right),\quad\\ z\geq \alpha_0
\end{array}
\end{equation}
satisfies  
\begin{equation}\label{condepsilon2}
\varepsilon_2(\alpha_0)=2\widetilde{\Phi}_2(+\infty,\alpha_0)+\dfrac{2L_M\left[\mathcal{G}\left( a\alpha_0 \tfrac{\sqrt{N_m}}{\sqrt{L_M}}\right) \frac{L_M^2}{2L_m^2N_m} +\widetilde{L}_2\right]}{L_m^2 \left[ 1-\mathcal{G}\left( a\alpha_0 \tfrac{\sqrt{N_M}}{\sqrt{L_m}}\right)\right]}\left(1+
\dfrac{k^2K_M}{16a^2\pi^{2}L_m}\dfrac{1}{\alpha_0^2} \right)<1,
\end{equation} where $\mathcal{G}(x):=\sqrt{\pi} x \exp(x^2)(1-\erf(x)),\; x>0$.

Then, for each $\alpha_0<\xi<\overline{\xi}_2$,  the 
self-map operator $W: \mathcal{K}\to \mathcal{K}$ defined by \eqref{41} is a contraction operator
 where $\overline{\xi}_2>0$ is defined as a unique solution to $\varepsilon_2(z)=1$.
\end{lemma}

\begin{proof}
Notice that for $\eta\geq \xi$ we have 
\begin{equation}
\begin{array}{l}
\Bigg|W(u_2)(\eta)-W(u_2^*)(\eta)\Bigg|\leq  \bigg|\Phi_2(\eta,u_2^*,\xi)-\Phi_2(\eta,u_2,\xi)\bigg|+\bigg|\dfrac{\chi_2(\eta,u_2^*,\xi)}{\chi_2(+\infty,u_2^*,\xi)}-\dfrac{\chi_2(\eta,u_2,\xi)}{\chi_2(+\infty,u_2,\xi)}\bigg|\\ \\
+\bigg|\Phi_2(+\infty,u_2,\xi)\dfrac{\chi_2(\eta,u_2,\xi)}{\chi_2(+\infty,u_2,\xi)}-\Phi_2(+\infty,u_2^*,\xi)\dfrac{\chi_2(\eta,u_2^*,\xi)}{\chi_2(+\infty,u_2^*,\xi)}\bigg|:= T_1+T_2+T_3
\end{array}
\end{equation}
where
$$T_1\leq\widetilde{\Phi}_2(\eta,\xi)||u_2-u_2^*||\leq\widetilde{\Phi}_2(+\infty,\alpha_0)||u_2-u_2^*||,$$
\vspace{0.25cm}
$$T_2\leq \dfrac{|\chi_2(\eta,u_2^*,\xi)-\chi_2(\eta,u_2,\xi)|}{|\chi_2(+\infty,u_2^*,\xi)|}+\dfrac{|\chi_2(+\infty,u_2,\xi)-\chi_2(+\infty,u_2^*,\xi)|}{|\chi_2(+\infty,u_2^*,\xi)||\chi_2(+\infty,u_2,\xi)|}|\chi_2(\eta,u_2,\xi)|$$
$$\leq \tfrac{ 2 L_M\sqrt{L_m}}{a \sqrt{N_M}}\tfrac{\widetilde{\chi}_2 (+\infty,\xi) \exp\bigg(-a^2\tfrac{N_M}{L_m}\xi^2\bigg)}{h(\infty,N_M,L_m,\xi) } ||u_2-u_2^*|| \leq \tfrac{2L_M\left[\mathcal{G}\left( a\xi \tfrac{\sqrt{N_m}}{\sqrt{L_M}}\right) \frac{L_M^2}{2L_m^2N_m} +\widetilde{L}_2\right]}{L_m^2 \left[ 1-\mathcal{G}\left( a\xi \tfrac{\sqrt{N_M}}{\sqrt{L_m}}\right)\right]}||u_2-u_2^*||$$
where $h$, $\widetilde{\chi}_2$ and $\widetilde{\Phi}_2$ are given by \eqref{h}, \eqref{tildechi2} and \eqref{phi2tilde}, respectively. In addition,
$$T_3\leq \bigg|\Phi_2(+\infty,u_2,\xi)-\Phi_2(+\infty,u_2^*,\xi)\bigg| \dfrac{\chi_2(\eta,u_2,\xi)}{\chi_2(+\infty,u_2,\xi)}+\Phi_2(+\infty,u_2^*,\xi)\bigg|\tfrac{\chi_2(\eta,u_2,\xi)}{\chi_2(+\infty,u_2,\xi)}-\tfrac{\chi_2(\eta,u_2^*,\xi)}{\chi_2(+\infty,u_2^*,\xi)}\bigg|$$
$$\leq \left(\widetilde{\Phi}_2(+\infty,\alpha_0)+\tfrac{k^2K_M}{16a^2\pi^{2}L_m\alpha_0^2}\tfrac{2L_M\left[\mathcal{G}\left( a\xi \tfrac{\sqrt{N_m}}{\sqrt{L_M}}\right) \frac{L_M^2}{2L_m^2N_m} +\widetilde{L}_2\right]}{L_m^2 \left[ 1-\mathcal{G}\left( a\xi \tfrac{\sqrt{N_M}}{\sqrt{L_m}}\right)\right]} \right)||u_2-u_2^*||.$$

Finally, we get
$$||W(u_2)-W(u_2^*)||\leq\varepsilon_2(\xi)||u_2-u_2^*||.$$

Taking into account that $\mathcal{G}$ is an increasing function that satisfies $0\leq \mathcal{G}(x)<1$, we obtain that $\varepsilon_2$ is an increasing function that satisfies $\varepsilon_2(+\infty)=+\infty$. 
If $\varepsilon_2(\alpha_0)<1$, there exists $\overline{\xi}_2> \alpha_0$ such that $\varepsilon_2(\overline{\xi}_2)=1$. Therefore,  $W$ is a contraction operator for $\alpha_0<\xi<\overline{\xi}_2$.
\end{proof}

\begin{theorem}\label{th2-a}
Suppose \eqref{44}-\eqref{55} and \eqref{condepsilon2} hold. Then, for each $\alpha_0<\xi<\Bar{\xi}_2$ there exists a unique solution $u_2\in \mathcal{K}$ to integral equation \eqref{34}.
\end{theorem}

\begin{proof}
It follows immediately from Lemma \ref{th2}.
\end{proof}

According to Theorems \ref{th1-a} and \ref{th2-a}, for each $\alpha_0<\xi<\hat{\xi}:=\min(\overline{\xi}_1,\overline{\xi}_2)$ we have unique solutions $u_1=u_1(\xi)$ and $u_2=u_2(\xi)$ to the equations \eqref{33} and \eqref{34}, respectively.

It remains to be proved that there exists a solution $\xi\in (\alpha_0,\hat{\xi})$ to the equation \eqref{39}.

Let us define the following function 
\begin{equation}
  \mathcal{Z}(\xi):=    \frac{-1+\Phi_2(\infty,u_2(\xi),\xi)}{\chi_2(\infty,u_2(\xi),\xi)}+Q(\xi,u_1(\xi),\alpha_0)E_1(\xi,u_1(\xi),\alpha_0).
\end{equation}

\begin{theorem} \label{existencia-xi}
Assume that \eqref{44}-\eqref{55} and \eqref{condepsilon2} hold. Then, for each $\alpha_0<\xi<\hat{\xi}$ we have that
\begin{equation}
    \mathcal{Z}_2(\xi)\leq \mathcal{Z}(\xi)\leq \mathcal{Z}_1(\xi)
\end{equation}
where
\begin{equation}
    \begin{array}{ll}
    \mathcal{Z}_1(\xi):= \alpha_0^2 P^*+\frac{K_M k^2}{16 a^2 \pi^2 L_m} \Bigg(\frac{1}{\alpha_0}+\dfrac{L_M}{\xi\; \left(1-\mathcal{G}\left(a\hat{\xi} \;\tfrac{\sqrt{N_M}}{ \sqrt{L_m}} \right) \right)}\Bigg),\\ \\
     \mathcal{Z}_2(\xi):= \alpha_0 P^*\exp\left( a^2\tfrac{N_M}{L_m}(\xi^2-\alpha_0^2)\right)-\dfrac{L_M \;\xi}{1-\mathcal{G}\left(a\xi \;\tfrac{\sqrt{N_M}}{ \sqrt{L_m}} \right) }.
    \end{array}
\end{equation}
Moreover, if $\mathcal{Z}_1(\hat{\xi})\leq M \hat{\xi}^3$ and   $\mathcal{Z}_2(\alpha_0)\geq M \alpha_0^3$, this is 
\begin{equation}\label{cota-Z1-xiraya}
     \alpha_0^2 P^*+\frac{K_M k^2}{16 a^2 \pi^2 L_m} \Bigg(\frac{1}{\alpha_0}+\frac{L_M}{\hat{\xi} \left(1-\mathcal{G}\left(a\hat{\xi} \;\tfrac{\sqrt{N_M}}{ \sqrt{L_m}} \right) \right)}\Bigg) \leq M \hat{\xi}^3
\end{equation}
and
\begin{equation}\label{cota-Z2-alpha0}
    \alpha_0 P^*-\frac{L_M \alpha_0}{1-\mathcal{G}\left(a\alpha_0 \;\tfrac{\sqrt{N_M}}{ \sqrt{L_m}} \right) }\geq M \alpha_0^3
\end{equation}
there exists at least one solution $\xi^*\in (\alpha_0,\hat{\xi})$   to the equation \eqref{39}.
\end{theorem}
\begin{proof}
From the definition of $Q$ given by \eqref{Q} and the bounds for $E_1$ we obtain that
\begin{equation}\label{cotasup-term1}
  Q(\xi,u_1,\alpha_0)E_1(\xi,u_1,\alpha_0)\leq \left(\alpha_0^2 P^*+\frac{K_M k^2}{16a^2 \pi^2 L_m} \right) \exp\left(-a^2 \tfrac{N_m}{L_M}(\xi^2-\alpha_0^2) \right)
\end{equation}
and
\begin{equation}\label{cotainf-term1}
   Q(\xi,u_1,\alpha_0)E_1(\xi,u_1,\alpha_0)\geq    \alpha_0^2 P^*   \exp\left(-a^2 \tfrac{N_m}{L_M}(\xi^2-\alpha_0^2) \right).
\end{equation}

In addition, according to \eqref{59} and \eqref{61} we have  
\begin{equation}\label{cotasup-term2}
\dfrac{-1+\Phi_2(\infty,u_2,\xi)}{\chi_2(\infty,u_2,\xi)}\leq \dfrac{\Phi_2(\infty,u_2,\xi)}{\chi_2(\infty,u_2,\xi)}
\leq \dfrac{K_M k^2 }{16a^2 \pi^2 L_m} \dfrac{L_M}{\xi \left(1-\mathcal{G}\left(a\hat{\xi} \;\tfrac{\sqrt{N_M}}{ \sqrt{L_m}} \right) \right)}.
\end{equation}
Moreover,
\begin{equation}\label{cotasinf-term2}
\begin{array}{ll}
\dfrac{-1+\Phi_2(\infty,u_2,\xi)}{\chi_2(\infty,u_2,\xi)}\geq \dfrac{-1}{\chi_2(\infty,u_2,\xi)}\geq \dfrac{-L_M \xi}{1-\mathcal{G}\left(a\xi \;\tfrac{\sqrt{N_M}}{ \sqrt{L_m}} \right)}
\end{array}
\end{equation}
Therefore, from  \eqref{cotasup-term1}, \eqref{cotainf-term1}, \eqref{cotasup-term2} and \eqref{cotasinf-term2} we get that the function $\mathcal{Z}$ is bounded by the functions  $\mathcal{Z}_1$ and $\mathcal{Z}_2$, which are independent of  $u_1(\xi)$ and $u_2(\xi)$.

Finally, taking into account that $\mathcal{Z}_1$ and $\mathcal{Z}_2$ are decreasing functions, if we assume that $\mathcal{Z}_1(\hat{\xi})\leq M \hat{\xi}^3$ and   $\mathcal{Z}_2(\alpha_0)\geq M \alpha_0^3$ we can guarantee that there exists at least one solution $\xi^*\in (\alpha_0,\hat{\xi})$ to the equation \eqref{39}.
\end{proof}

\begin{theorem} Suppose \eqref{alpha}, \eqref{condP}, \eqref{alfacero}, \eqref{44}-  \eqref{55}, \eqref{condepsilon2}, \eqref{cota-Z1-xiraya} and \eqref{cota-Z2-alpha0} hold 
then there exists at least one  solution $(u_1^*,u_2^*,\xi^*)$ with $\alpha_0<\xi^*<\hat{\xi}$ to the problem \eqref{33}-\eqref{39}, where $\alpha_0$ is given by \eqref{alfacero}.

\end{theorem}

\begin{proof}
It follows immediately from Theorems \ref{th1-a}, \ref{th2-a} and \ref{existencia-xi} . 
\end{proof}

Finally, we can enunciate the following theorem:
\begin{theorem} {\color{blue} If} \eqref{condP}, \eqref{alfacero}, \eqref{44}-\eqref{55}, \eqref{condepsilon2}, \eqref{cota-Z1-xiraya} and \eqref{cota-Z2-alpha0} hold 
there exists at least one  solution $(\theta_1^*(r,t),\theta_2^*(r,t),\alpha(t),\beta(t))$ to the problem \eqref{6}-\eqref{13}, where 
$$\theta_1^{*}(r,t)=\theta_mu_1^*\left(\frac{r}{2a\sqrt{t}}\right)+\theta_m, \quad \alpha(t)\leq r\leq\beta(t), \quad t>0$$.
$$\theta_2^{*}(r,t)=\theta_mu_2^*\left(\frac{r}{2a\sqrt{t}}\right)+\theta_m, \quad r\geq\beta(t), \quad t>0$$
$$\alpha(t)=2a\alpha_0\sqrt{t},\quad t>0\quad \textit{with}\quad \alpha_0=\frac{2\sqrt{\pi}\lambda_b(\theta_{ion}-\theta_b)}{P-\sqrt{P^{2}-8a^{2}L_b\gamma_b\lambda_b\pi(\theta_{ion}-\theta_b)}}$$
and 
$$\beta(t)=2a\xi^{*}\sqrt{t}, \quad t>0\quad \textit{where}\quad \xi^{*}\quad \textit{ is a solution of} \quad \eqref{39}$$
\end{theorem}
\section{Conclusion}
In conclusion, this paper has addressed the problem of vaporization and melting in closed electrical contacts. The significance of this problem for engineers working with electrical contact materials has been highlighted. The solution approach involves solving the spherical heat equation, taking into account the temperature-dependent thermal coefficients and the effect of Joule heating.

The mathematical model developed in this study provides a comprehensive description of the vaporization and melting process in closed electrical contacts. By considering three zones and incorporating temperature-dependent thermal coefficients and Joule heating, the model captures the essential factors influencing the behavior of the materials.

Our analysis has demonstrated the possibility of obtaining a solution to the problem through similarity transformation, resulting in a nonlinear integral equation. This transformation greatly simplifies the solution approach and facilitates further analysis of the problem. The relevance and applicability of this integral equation have been established.

We have successfully discussed and proved the existence and uniqueness of the solution to the integral equation obtained through similarity transformation. This finding validates the reliability and robustness of the solution method employed in this study. It provides a solid basis for obtaining meaningful results and insights into the vaporization and melting process in closed electrical contacts.

The results of this study have practical implications for engineers working with electrical contact materials. The mathematical model developed in this research can be effectively employed to analyze these materials, particularly during the initial stages of arcing where metal ignition occurs. The nonlinear model proposed in this paper proves to be highly effective in capturing the complex behavior of the system. These findings contribute to a better understanding of the phenomenon and provide valuable insights for engineering applications.

\section*{Acknowledgment}
\noindent T.A. Nauryz and S.N. Kharin were supported by grant project AP19675480 "Problems of heat conduction with a free boundary arising in modeling of switching processes in electrical devices" from Ministry of Sciences and Higher Education of the Republic of Kazakhstan. J. Bollati and A. C. Briozzo were supported by the Projects 80020210100002 and 80020210200003 from Austral University, Rosario, Argentina.

\section*{References}
 

\end{document}